\documentclass[12pt]{amsart}

\usepackage{fix-cm}
\usepackage[normalem]{ulem}

\usepackage{lineno}

\usepackage{soul}

\usepackage{graphicx}

\usepackage{amssymb, amsmath, amsthm, hyperref, euscript, color}

\usepackage[dvipsnames]{xcolor}

\usepackage{amscd}

\usepackage{tikz-cd}

\usepackage{bbm}

\usepackage{enumerate}

\usepackage[english]{babel}

\usepackage{mathtools}

\usepackage{makecell}

\numberwithin{equation}{section}

\newtheoremstyle{theor}{6pt plus 1pt minus 1pt}{6pt plus 1pt minus 1pt}{\slshape}{}{\bfseries}{.}{5pt plus 1pt minus 1pt}{}
\newtheoremstyle{def}{6pt plus 1pt minus 1pt}{6pt plus 1pt minus 1pt}{}{}{\bfseries}{.}{5pt plus 1pt minus 1pt}{}
\newtheoremstyle{rmk}{6pt plus 1pt minus 1pt}{6pt plus 1pt minus 1pt}{}{}{\bfseries}{.}{5pt plus 1pt minus 1pt}{}
\newtheoremstyle{claim}{6pt plus 1pt minus 1pt}{6pt plus 1pt minus 1pt}{}{}{\bfseries}{.}{5pt plus 1pt minus 1pt}{}

\theoremstyle{theor}
\newtheorem{newstatement}{newstatement}
\newtheorem{lemma}[newstatement]{Lemma}
\newtheorem{theorem}[newstatement]{Theorem}
\newtheorem*{theorem*}{Theorem 2}

\newtheorem{proposition}[newstatement]{Proposition}

\theoremstyle{def}
\newtheorem{definition}[newstatement]{Definition}

\theoremstyle{rmk}
\newtheorem{remark}[newstatement]{Remark}

\newtheorem*{example*}{Example}

\theoremstyle{claim}

\theoremstyle{theor}
\newtheorem{thm}{Theorem}

\theoremstyle{goa}

\newtheorem{con}[thm]{Construction}

\newtheorem{dfs}[thm]{Definition}

\newtheorem{qst}[thm]{Question}

\expandafter\let\expandafter\oldproof\csname\string\proof\endcsname
\let\oldendproof\endproof
\renewenvironment{proof}[1][\proofname]{%
  \oldproof[\slshape #1]%
}{\oldendproof}
\def\provedboxcontents#1{$\square$}
\makeatletter
\newsavebox\myboxA
\newsavebox\myboxB
\newlength\mylenA

\newcommand*\xoverline[2][0.75]{%
    \sbox{\myboxA}{$\m@th#2$}%
    \setbox\myboxB\null% Phantom box
    \ht\myboxB=\ht\myboxA%
    \dp\myboxB=\dp\myboxA%
    \wd\myboxB=#1\wd\myboxA% Scale phantom
    \sbox\myboxB{$\m@th\overline{\copy\myboxB}$}%  Overlined phantom
    \setlength\mylenA{\the\wd\myboxA}%   calc width diff
    \addtolength\mylenA{-\the\wd\myboxB}%
    \ifdim\wd\myboxB<\wd\myboxA%
       \rlap{\hskip 0.5\mylenA\usebox\myboxB}{\usebox\myboxA}%
    \else
        \hskip -0.5\mylenA\rlap{\usebox\myboxA}{\hskip 0.5\mylenA\usebox\myboxB}%
    \fi}
\makeatother
\newcommand{\Q}{\mathbb{Q}}

\newcommand{\R}{\mathbb{R}}

\newcommand{\Z}{\mathbb{Z}}

\newcommand{\N}{\mathbb{N}}

\newcommand{\CP}{{\mathbb C\mkern-0.5mu\mathrm P}}

\DeclareMathOperator{\id}{id}

\DeclareMathOperator{\Ks}{KS}

\newcommand{\RP}{{\mathbb R\mkern-0.5mu\mathrm P}}

\newcommand{\simtimes}{\mathbin{\widetilde{\smash{\times}}}}

\DeclareMathOperator{\Pin}{Pin}
\DeclareMathOperator{\TopPin}{TopPin}

\DeclareMathOperator{\Sp}{Spin}

\DeclareMathOperator{\Ar}{arf}
\newcommand{\cs}{\mathbin{\#}}

\makeatletter
\@namedef{subjclassname@2020}{%
  \textup{2020} Mathematics Subject Classification}
\makeatother

\begin{document}

\author{Rafael Torres}

\title[Nonorientable 4-manifolds with cyclic $\pi_1$.]{Topological classification of certain nonorientable 4-manifolds with cyclic fundamental group of order 2 mod 4.}

\address{Scuola Internazionale Superiori di Studi Avanzati (SISSA)\\ Via Bonomea 265\\34136\\Trieste\\Italy}

\email{rtorres@sissa.it}


\maketitle

\emph{Abstract}: The classification up to homeomorphism of closed topological nonorientable 4-manifolds with fundamental group of order 2 due to Hambleton-Kreck-Teichner can be used to classify a large set of 4-manifolds with $\pi_1 = \Z/2p$ for odd $p > 1$. It is plausible that this set comprises all homeomorphism classes of such 4-manifolds. This is done through a cut-and-paste construction, which is then used to calculate an extension of the Hausmann-Weinberger invariant for every cyclic group of even order and to produce 4-manifolds that are not homotopy equivalent to a connected sum and which need not be smoothable.

\section{Introduction.}\label{Introduction}

There is a striking difference between our current understanding of closed topological nonorientable 4-manifolds with respect to the oriented ones. Instances of such imbalance in the literature include\begin{itemize}\item classification results up to homeomorphism \cite{[FreedmanQuinn], [FNOP]},\item computations of minimizers of the Euler characteristic with a prescribed fundamental group \cite{[HausmannWeinberger], [KirkLivingston], [Kotschick]} and
\item examples of non-smoothable 4-manifolds with trivial Kirby-Siebenmann invariant \cite{[FriedlHambletonMelvinTeichner]}.\end{itemize} This note contains three observations whose purpose is to decrease the imbalance in these three instances as we now elaborate. 

Hambleton-Kreck-Teichner obtained a classification up to homeomorphism of closed topological nonorientable 4-manifolds with fundamental group of order two in \cite[Theorems 1 and 2]{[HKT]}, where they also provided an explicit list of examples that realize every homeomorphsim class \cite[Theorem 3]{[HKT]}. They showed that any homeomorphism class of such a 4-manifold $X$ is determined by\begin{itemize}
\item the (non-)existence of a $\TopPin^\pm$-structure on it,
\item its Euler characteristic $\chi(X)$, 
\item the Stiefel-Whitney number $w_1(X)^4$, 
\item its Kirby-Siebenmann invariant $\Ks(X)$ and
\item a $\Z/8$-valued Arf invariant $\Ar(X, \Phi)$ for a primitive $\TopPin^c$-structure $(X, \Phi_X)$.
\end{itemize}

The first of the three observations contained in this paper is an extension of Hambleton-Kreck-Teichner's results to 4-manifolds with cyclic fundamental group of order $2p$ for any odd $p > 1$ through a construction introduced in \cite{[BaisTorres]}, which we now recall and generalize. The nonorientable 3-disk bundle over the circle is denoted by $D^3\widetilde{\times} S^1$ and the nonorientable 2-sphere bundle over the circle by $S^2\widetilde{\times} S^1$.

\begin{con}\label{Construction Main}Take the nonorientable 3-sphere bundle over the circle $S^3\simtimes S^1$ and let $\alpha\subset S^3\simtimes S^1$ be a locally flat simple loop whose homotopy class generates the group $\pi_1(S^3\simtimes S^1) = \Z$. Let $p > 1$ be an integer and denote by $\alpha^p\subset S^3\simtimes S^1$ the locally flat simple loop that represents the homology class $[\alpha]^p\in H_1(S^3\simtimes S^1; \Z) = \Z$.

$\bullet$ If $p = 0 \mod 2$, then $\alpha^p\subset S^3\widetilde{\times} S^1$ is an orientation-preserving loop with trivial tubular neighborhood $\nu(\alpha^p) = S^1\times D^3$ with boundary $\partial \nu(\alpha^p)) = S^1\times S^2$.

$\bullet$ If $p\neq 0 \mod 2$, then $\alpha^p\subset S^3\widetilde{\times} S^1$ is an orientation-reversing simple loop whose tubular neighborhood is $\nu(\alpha^p) = D^3\simtimes S^1$ with boundary $\partial \nu(\alpha^p) = S^2\widetilde{\times} S^1$; see \cite[Section 9]{[FreedmanQuinn]} and \cite[Theorems 5.5 and 5.6]{[FNOP]} for results on existence and uniqueness of tubular neighborhoods in the topological category. 

Guided by the parity of $p$, construct the compact topological nonorientable 4-manifolds\begin{equation}\label{Nonorientable Piece 1}
S_p=(S^3 \simtimes S^1) \setminus \nu(\alpha^p)\end{equation}with boundary $\partial S_p = S^1\times S^2$ and\begin{equation}\label{Nonorientable Piece}
N_p=(S^3 \simtimes S^1) \setminus \nu(\alpha^p)
\end{equation}with boundary $\partial N_p = S^2\simtimes S^1$.\end{con}

The aforementioned extension of Hambleton-Kreck-Teichner's result is as follows. Take a closed topological nonorientable 4-manifold $X$with $\pi_1(X) = \Z/2$ and let $\alpha_X\subset X$ be a locally flat closed simple loop whose homotopy class generates the fundamental group of $X$. Assemble the closed topological nonorientable 4-manifold\begin{equation}\label{Nonorientable Construcion 2p}X_{2p}:= (X\setminus \nu(\alpha))\cup N_p.\end{equation}We omit the choice of gluing homeomorphism in (\ref{Nonorientable Construcion 2p}) from the notation since any self-homeomorphism of $\partial(X\setminus \nu(\alpha)) = S^2\widetilde{\times} S^1$ extends to $N_p$ by \cite[Theorem A]{[BaisTorres]}. The Seifert-van Kampen theorem implies that the 4-manifold (\ref{Nonorientable Construcion 2p}) has fundamental group $\Z/2p$. We show in Section \ref{Section Invariants} that the set of invariants $\{\chi, w_1^4, \Ks, \Ar\}$ is unchanged by Construction \ref{Construction Main}. We gather these observations the next theorem. If the 4-manifold $M$ admits a smooth structure, then we denote by $\eta(M,\Phi_M)$ the $\eta$-invariant for a $\Pin^c$-structure $(M, \Phi_M)$ as defined by Gilkey \cite{[Gilkey]}.

\begin{thm}\label{Theorem Homeomorphism Invariants}Let $X_{2p}$ and $X_{2p}'$ be closed topological nonorientable 4-manifolds that are obtained by applying Construction \ref{Construction Main} to $X$ and $X'$, respectively.

$\bullet$ There is a homeomorphism\begin{equation}\label{Homeomorphism}X_{2p}\rightarrow X_{2p}'\end{equation}if and only if there is a homeomorphism\begin{equation}\label{Homeomorphism Data}X\rightarrow X',\end{equation}where $X$ and $X'$ are the initial data used in Construction \ref{Construction Main} to obtain $X_{2p}$ and $X_{2p}'$, respectively.

$\bullet$ There is a homeomorphism (\ref{Homeomorphism}) if and only if
\begin{enumerate}
\item $\chi(X_{2p}) = \chi(X_{2p}')$,
\item $w_1(X_{2p})^4 = w_1(X_{2p}')^4$,
\item $\Ks(X_{2p}) = \Ks(X_{2p}')$, and
\item $\Ar(X_{2p}, \Phi_{X_{2p}})= \pm \Ar(X_{2p}', \Phi_{X'_{2p}})$
\end{enumerate}
for some primitive $\TopPin^c$-structures $(X_{2p}, \Phi_{X_{2p}})$ and $(X_{2p}', \Phi_{X'_{2p}})$.

$\bullet$ If $X_{2p}$ and $X_{2p}'$ admit a smooth structure, then the homeomorphism (\ref{Homeomorphism}) exists if and only if the identities of Items (1) and (2) are satisfied as well as the identity\begin{equation*}\eta(X_{2p}, \Phi_{X_{2p}}) = \pm \eta(X_{2p}', \Phi_{X'_{2p}})\end{equation*}for some primitive $\Pin^c$-structures $(X_{2p}, \Phi_{X_{2p}})$ and $(X_{2p}', \Phi_{X'_{2p}})$.
\end{thm}

%\begin{thm}\label{Theorem Equivalence}Let $X_{2p}$ and $X_{2p}'$ be closed topological nonorientable 4-manifolds with finite cyclic fundamental group of order 2p for $p > 1$ odd that are contained in Collection \ref{Collection Manifolds}. There are closed topological nonorientable 4-manifolds $X$ and $X'$ with fundamental group of order 2 that are the initial data in the construction of $X_{2p}$ and $X_{2p}'$, respectively, through Construction \ref{Construction Main}. Moreover, the statements\begin{enumerate}[(i)]
%\item There is a homeomorphism $X\rightarrow X'$.
%\item There is a homeomorphism $X_{2p}\rightarrow X'_{2p}$. 
%\end{enumerate}are equivalent.
%\end{thm}

Much like the Arf invariant $\Ar(M, \Phi_M)$ is a $\TopPin^c$-bordism invariant \cite{[HKT]},  $\eta(M, \Phi_M)$ is a $\Pin^c$-bordism invariant \cite{[Gilkey], [BahriGilkey]}. A formula that relates these two invariants was obtained by Hambleton-Kreck-Teichner \cite[Theorem 4]{[HKT]} and it is used in the proof of Theorem \ref{Theorem Homeomorphism Invariants}. %Notice that Theorem \ref{Theorem Homeomorphism Invariants} enables the use of the results in \cite{[HKT]} to classify the 4-manifolds arising from Construction \ref{Construction Main}. 

%The aforementioned work of Hambleton-Kreck-Teichner and work of Ruberman-Stern \cite{[RubermanStern]} show that a closed topological nonorientable 4-manifold with fundamental group of order 2 admits a smooth structure if and only if its Kirby-Siebenmann invariant vanishes; see \cite{[KirbySiebenmann]} for further details on this $\Z/2$-valued invariant. In this direction, we prove the following result. 

%\begin{thm}\label{Theorem KS}A closed topological nonorientable 4-manifold that is contained in Collection \ref{Collection Manifolds} admits a smooth structure if and only if its Kirby-Siebenmann invariant vanishes. 
%\end{thm}

%The results presented so far raise the following questions.

%\begin{question}\label{Question All} Is Collection \ref{Collection Manifolds} a complete list of representatives of all homeomorphism classes of closed nonorientable topological 4-manifolds with cyclic fundamental group of order $2p$ for odd $p > 1$?
%\end{question}

\begin{qst}\label{Question All'}Does the set of invariants\begin{equation*}\{\chi, w_1^4, \Ks, \Ar\}\end{equation*}along with the (non-)existence of a $\TopPin^\dagger$-structure for $\dagger \in \{+, -, c\}$ classify closed topological nonorientable 4-manifolds with cyclic fundamental group of order $2p$ for odd $p > 1$ up to homeomorphism?
\end{qst}

%The second observation is based on the following cut-and-paste construction of nonorientable 4-manifolds from orientable ones.

%\begin{con}\label{Construction Main 2} Let $Y$ be a closed topological orientable 4-manifold and let $\alpha_Y\subset Y$ be a locally flat simple loop.
%Construct the closed topological nonorientable 4-manifold\begin{equation}\label{Manifold Construction 2}Z(\varphi) = (Y\setminus \overset{\circ}{\nu}(\alpha_Y))\cup_{\varphi}(Mb\times S^2)\end{equation}with gluing homeomorphism\begin{equation}\varphi: \partial(Y\setminus \overset{\circ}{\nu}(\alpha_Y)\rightarrow \partial Mb\times S^2 = S^1\times S^2.\end{equation}

%The orientation double cover of (\ref{Manifold Construction 2}) is given by\begin{equation}\widehat{Z}(\varphi) = (Y\setminus \overset{\circ}{\nu}(\alpha_Y))\cup_{\varphi}(\overline{Y\setminus \overset{\circ}{\nu}(\alpha_Y)})\end{equation}
%\end{con}

The second observation highlights the utility of Construction \ref{Construction Main} to produce 4-manifolds with minimal Euler characteristic and study the existence landscape of nonorientable 4-manifolds. To the best of our knowledge, there is no systematic study of these topics in the literature in glaring contrast to the knowledge available in the orientable case. Hausmann-Weinberger defined an invariant of a finitely presented group $G$ \cite{[HausmannWeinberger]} as the minimizer of the Euler characteristic of closed oriented 4-manifolds with fundamental group $G$. This is a considerably difficult invariant to compute and it is only known in a handful of cases \cite{[HausmannWeinberger], [KirkLivingston], [Kotschick]}. We sample the utility of Construction \ref{Construction Main} with an extension of the Hausmann-Weinberger invariant to nonorientable 4-manifolds.  Recall that for any finitely presented group $G$ with a subgroup of index two, there is a closed nonorientable topological 4-manifold whose fundamental group is isomorphic to $G$ \cite{[Kreck], [MantioneTorres]}. We introduce the following extension of this invariant  to nonorientable 4-manifolds. 

\begin{dfs}\label{Definition HW Invariant}Let $G$ be a finitely presented group. The nonorientable Hausmann-Weinberger invariant $q(G)$ of $G$ is the minimal value of the Euler characteristic of a closed nonorientable topological 4-manifold $M$ with fundamental group $G$.
\end{dfs}

Implicit in the results of Hambleton-Kreck-Teichner mentioned above and Hambleton-Hillman \cite{[HambletonHillman]} are the values\begin{center}$q(\Z/2) = 1 = q(\Z/4)$.\end{center} The building blocks in Construction \ref{Construction Main} allows us to establish the value of the invariant of Definition \ref{Definition HW Invariant} for finite cyclic groups.

\begin{thm}\label{Theorem HW}The nonorientable Hausmann-Weinberger invariant for finite cyclic groups is $q(\Z/2p) = 1$ for $p\in \N$.
\end{thm}

Recall that no group of odd order is the fundamental group of a closed nonorientable 4-manifold.

Hambleton-Teichner \cite{[HambletonTeichner]} used a result of Freedman-Quinn \cite{[FreedmanQuinn]} to produce a closed oriented topological 4-manifold $Y$ with infinite cyclic fundamental group, Euler characteristic $\chi(Z) = 4$ and Kirby-Siebenmann invariant $\Ks(Y) = 0$ that is not homotopy equivalent to a connected sum $Y_0\cs (S^1\times S^3)$ for a simply connected 4-manifold $Y_0$. The third observation is that gluing a building block taken from Construction \ref{Construction Main} to $Y$ yields the following result.

\begin{thm}\label{Theorem Candidate}There is a closed nonorientable topological 4-manifold $Y'$ with fundamental group $\pi_1(Y') = \Z$, Euler characteristic $\chi(Y') = 4$ and Kirby-Siebenmann invariant $\Ks(Y') = 0$ that is not homotopy equivalent to a connected sum $Y_0\cs(S^3\widetilde{\times} S^1)$ for any closed simply connected 4-manifold $Y_0$ and whose orientation 2-cover is the connected sum $4(\CP^2\cs \overline{\CP^2})\cs(S^1\times S^3)$.
\end{thm}

A similar existence result can be proven using algebraic machinery of Hambleton-Teichner \cite{[HambletonTeichner]} and surgery theory-based results of Wang \cite{[Wang]}. Our proof of Theorem \ref{Theorem Candidate} is explicit and hands on. Friedl-Hambleton-Melvin-Teichner \cite{[FriedlHambletonMelvinTeichner]} showed that the oriented 4-manifold $Y$ does not admit a smooth structure. This prompts the following question.

\begin{qst}\label{Question Smooth}Does the 4-manifold $Y'$ of Theorem \ref{Theorem Candidate} admit a smooth structure? 
\end{qst}

Either answer to Question \ref{Question Smooth} unveils new phenomena in the nonorientable realm. If $Y'$ were to be smoothable, then it would be a nonorientable counter-example to a known conjecture that states that a closed smooth 4-manifold with infinite cyclic fundamental group is homeomorphic to a connected sum \cite[Conjecture 1.3]{[FriedlHambletonMelvinTeichner]}. If $Y'$ were to be non-smoothable, then it follows that there is a free non-smoothable topological $\Z/2$-action on $4(\CP^2\cs \overline{\CP^2})\cs(S^1\times S^3)$ whose quotient has zero Kirby-Siebenmann invariant; see Remark \ref{Remark Smoothable}.

This short paper has the following organization. The proofs of the results presented in the introduction are given in Section \ref{Section Proofs}. Section \ref{Section Pin Structures} contains a summary of existence and uniqueness results of $\TopPin^\dagger$-structures on 4-manifolds. The behavior of the basic topological invariants under Construction \ref{Construction Main} and some of its properties are studied in Section \ref{Section Invariants}. Theorem \ref{Theorem HW} is proven in Section \ref{Section Computation HW}. Theorem \ref{Theorem Candidate} is proven in Section \ref{Section Smoothability}.

\subsection{Acknowledgments} We thank Sebasti\'an Mat\'ias Camponovo, 

Daniel Kasprowski and Peter Teichner for useful conversations. We thank the MPI - Bonn for its hospitality while a portion of this work was being written.

\section{$\TopPin^{\dagger}$ and $\Pin^\dagger$-structures.}\label{Section Pin Structures}Our convention in this paper is to say that a 4-manifold admits a $\TopPin^\dagger$-structure if its topological tangent bundle admits a $\TopPin^\dagger$-structure for $\dagger = c, +, -$. If the underlying 4-manifold admits a smooth structure, our discussion centers around $\Pin^\dagger$-structures. The reader is directed to \cite{[KirbyTaylor]} and \cite[\S 2]{[HKT]} for background on these structures and to \cite{[FNOP]} for background on topological tangent and normal vector bundles. We first compare our convention with the one used in \cite{[HKT]}, where the authors equip the normal (vector) bundle with such structure.

\begin{lemma}Let $M$ be a topological 4-manifold. There is a 1-to-1 correspondence between $\TopPin^\pm$-structures on the topological tangent bundle of $M$ and $\TopPin^\mp$-structures on its normal vector bundle. 
\end{lemma}

Existence and classification results of these structures are contained in the following lemma.

\begin{lemma}\label{Lemma Pin}\cite[Lemma 1]{[HKT]}. Let $M$ be a topological manifold.  

\begin{itemize}

\item There is a $\TopPin^c$-structure $(M, \Phi^c_M)$ if and only if $\beta(w_2(M)) = 0$, where\begin{equation*}\beta:H^2(M; \Z/2)\rightarrow H^3(M; \Z)\end{equation*}is the Bockstein operator induced from the exact coefficient sequence $\Z\rightarrow \Z\rightarrow \Z/2$. 

\item $\TopPin^c$-structures $(M, \Phi^c_M)$ are in one-to-one correspondence with the second cohomology group $H^2(M; \Z)$. 

\item There is a $\TopPin^-$-structure $(M, \Phi^-_M)$ if and only if\begin{equation*}w_2(M) = w_1(M)^2.\end{equation*} 

\item There is a $\TopPin^+$-structure $(M, \Phi^+_M)$ if and only if $w_2(M) = 0$. 

\item $\TopPin^\pm$-structures $(M, \Phi^{\pm}_M)$ are in a one-to-one correspondence with the first cohomology grop $H^1(M; \Z/2)$.
\end{itemize}
\end{lemma}

%These structures provide the following trichotomy of 4-manifolds. 

%\begin{definition}A closed topological nonorientable 4-manifold is of\begin{itemize}
%\item $w_2$-type (I) if it does not admit any $\TopPin^{\pm}$-structure;
%\item $w_2$-type (II) if it admits a $\TopPin^-$-structure, and 
%\item $w_2$-type (III) if it admits a $\TopPin^+$-structure.\end{itemize} 
%\end{definition}

Lemma \ref{Lemma Pin} and the exact sequence\begin{equation*}\cdots\rightarrow H^2(\pi_1(M); \Z/2)\rightarrow H^2(M; \Z/2)\rightarrow H^2(\widehat{M}; \Z/2)\rightarrow \cdots\end{equation*}implies that $M$ admits a $\TopPin^\pm$-structure if the second Stiefel-Whitney class of its orientation 2-cover $\widehat{M}$ satisfies $w_2(\widehat{M}) = 0$. If the latter identity holds, then either $w_2(M) = w_1(M)^2$ or $w_2(M) = 0$. Moreover, any closed nonorientable topological 4-manifold with cyclic fundamental group admits a $\TopPin^c$-structure; cf. \cite[p. 653]{[HKT]}. The following definition is given in terms of $\Pin^c$-structures in \cite[\S 1]{[HKT]}.

\begin{definition}\label{Definition TopPin Primitive}A $\TopPin^c$-structure $(M, \Phi_M)$ is said to be primitive if either it comes from a $\TopPin^\pm$-structure or if the first Chern class $c_{\Phi_M}$ of the induced topological line bundle is a primitive cohomology class.
\end{definition}

The $\TopPin^c$-structure $(M, \Phi_M)$ is needed for the $\Ar$ invariant and, when $M$ has a smooth structure, a primitive $\Pin^c$-structure for the $\eta$-invariant. The reader might have noticed that there are instances where the notation used in \cite{[HKT]} for these invariants is actually $\Ar(X, c_{\Phi_M})$ and $\eta(M, c_{\Phi_M})$.

\section{Invariants of 4-manifolds obtained through Construction \ref{Construction Main}.}\label{Section Invariants}

\subsection{Case $p \neq 0 \mod 2$}This section is devoted to the computation of the invariants\begin{equation}\label{Invariants S2}\{\chi, w_1^4, \Ks, \Ar, \eta\}\end{equation} of the 4-manifolds obtained via Construction \ref{Construction Main} in terms of those an initial 4-manifold with fundamental group of order two. Throughout this section, we assume that $X$ is a closed topological nonorientable 4-manifold with $\pi_1(X) = \Z/2$ and $X_{2p}$ is the associated 4-manifold constructed in (\ref{Nonorientable Construcion 2p}) for odd values of $p$.

\begin{lemma}\label{Lemma Identities}The following identities hold.\begin{enumerate}

\item $\chi(X) = \chi(X_{2p})$.

\item $w_1(X)^4 = w_1(X_{2p})^4$.

\item $\Ks(X) = \Ks(X_{2p})$.

\item For a given primitive $\TopPin^c$-structure $(X, \Phi_{X})$, there are primitive $\TopPin^c$-structures $(X_{2p}, \Phi_{X_{2p}})$ and $(S^3\widetilde{\times} S^1, \Phi_{S^3\widetilde{\times} S^1})$ as well as a $\TopPin^c$-bordism\begin{equation}\label{Primitive Cobordism}((V, \Phi_V);, (X_{2p}, \Phi_{X_{2p}}), (X, \Phi_X)\sqcup (S^3\widetilde{\times} S^1, \Phi_{S^3\widetilde{\times} S^1})).\end{equation}

\item The 4-manifold $X_{2p}$ admits a $\TopPin^\dagger$-structure for $\dagger = \{+, -, c\}$ if and only if $X$ does. 
\end{enumerate}
\end{lemma}

\begin{proof}The claims in the first two items of Lemma \ref{Lemma Identities} are straight-forward and the details are left to the reader. Regarding the Kirby-Siebenmann invariant, there are two ways to argue. One way is to use \cite[Theorem 9.2 (5)]{[FNOP]}, which says\begin{equation}\Ks(X_{2p}) = \Ks(X\setminus \nu(\alpha)) + \Ks(N_p) = \Ks (X)\end{equation}since $\Ks(X\setminus \nu(\alpha)) = \Ks(X)$ and $\Ks(N_p) = \Ks(S^3\widetilde{\times} S^1) = 0$. The second one is to use results of Hsu \cite{[Hsu]} and Lashof-Taylor \cite{[LT]}, where the authors have shown that the Kirby-Siebenmann invariant is a bordism invariant cf. \cite{[KirbySiebenmann]}.

To address the claims of the fourth item, we assemble a primitive $\TopPin^c$-structure $(X_{2p}, \Phi_{X_{2p}})$ by equipping each of the building blocks in the decomposition $X_{2p} = (X\setminus \nu(\alpha))\cup N_p$ with matching $\TopPin^c$-structures. A prototype argument is to consider the case where the primitive $\TopPin^c$-structure arises from a $\TopPin^\pm$-structure; see Definition \ref{Definition TopPin Primitive}. A given $\TopPin^\pm$-structure $(X, \Phi_X)$ induces $\TopPin^\pm$-structures on the compact nonorientable 4-manifold $X\setminus \nu(\alpha_X)$ and on its boundary $\partial(X\setminus \nu(\alpha_X)) = S^2\widetilde{\times} S^1$ \cite[Construction 1.5]{[KirbyTaylor]}. Similarly, a $\TopPin^\pm$-structure $(S^3\widetilde{\times} S^1, \Phi_{S^3\widetilde{\times} S^1})$ induces a $\TopPin^\pm$-structure on the compact 4-manifold $(N_p, \Phi_{N_p})$ that restricts to a $\TopPin^\pm$-structure $(\partial N_p, \Phi_{N_p}|_\partial)$ on its boundary. We can glue the induced structures on $\partial N_p$ and $\partial (X\setminus \nu(\alpha))$ to produce a $\TopPin^\pm$-structure $(X_{2p}, \Phi_{X_{2p}})$.  The corresponding bordism (\ref{Primitive Cobordism}) is readily assembled by gluing $N_p$ to the top lid $X\times \{1\}$ of the product $X\times [0, 1]$ with matching $\TopPin^\pm$-structures. A similar argument establishes the $\TopPin^c$-case even in the abscence of a $\TopPin^\pm$-structure. Notice that the work of Hambleton-Kreck-Teichner lists canonical primitive $\TopPin^c$-structures for the initial data $X$ to be used in the assemblage of $X_{2p}$; see \cite[Theorem 3, \S 3]{[HKT]}.\end{proof}

The $\TopPin^c$-bordism of Lemma \ref{Lemma Identities} and the property of the Arf invariant and the $\eta$-invariant being $\TopPin^c$- and $\Pin^c$-bordism invariants (respectively) \cite{[BahriGilkey], [Gilkey], [HKT]} have the following implication. 

\begin{proposition}\label{Proposition Value Pin}For every primitive $\TopPin^c$-structure $(X, \Phi)$, there is a primitive $\TopPin^c$-structure $(X_{2p}, \Phi_{2p})$ such that\begin{equation}\Ar(X, \Phi) = \pm \Ar (X_{2p}, \Phi_{2p}).\end{equation}

For every primitive $\Pin^c$-structure $(X, \Phi)$, there is a primitive $\Pin^c$-structure $(X_{2p}, \Phi_{2p})$ such that\begin{equation}\eta(X, \Phi) = \pm \eta (X_{2p}, \Phi_{2p}).\end{equation} 
\end{proposition}

A topological formula for the $\eta$-invariant in terms of the $\Ar$ invariant is found in \cite[\S 5]{[HKT]}.

\subsection{Case $p = 0 \mod 2$}It was proven in \cite[Theorem A]{[BaisTorres]}, that any self-homeomorphism of $\partial N_p$ extends to a self-homeomorphism of $N_p$. The situation is entirely different if $p$ is an even number.

\begin{lemma}A self-diffeomorphism $\partial S_p\rightarrow \partial S_p$ need not extend to a self-diffeomorphism $S_p\rightarrow S_p$ if $p$ is even. 
\end{lemma}

\begin{proof}Notice that\begin{equation*}S_2 = (S^3\widetilde{\times} S^1\setminus \nu(\alpha^2)) = Mb\times S^2\end{equation*}and consider the homeomorphism $\varphi:S^1\times S^2\rightarrow S^1\times S^2$ given by $(\theta, x)\mapsto (\theta, \rho_\theta(x))$ for $(\theta, x)\in S^1\times S^2$ and $\rho_\theta:S^2\times S^2$ is a rotation of the 2-sphere through an axis by an angle $\theta\in S^1$ \cite{[Gluck]}. To see that this self-homeomorphism of $S^1\times S^2 = \partial S_2 = \partial Mb\times S^2$ does not extend to $S_2$ consider $\RP^2\times S^2 = (D^2\times S^2)\cup_{\id} (Mb\times S^2)$ and the nontrivial 2-sphere bundle over the real projective plane $S(2\gamma\oplus \R) = (D^2\times S^2)\cup_{\varphi}(Mb\times S^2)$. These two 4-manifolds have different second Stiefel-Whitney class and, hence, they are not homotopy equivalent; see \cite{[HKT], [HambletonHillman]}.\end{proof}

\begin{remark}A contribution of Construction \ref{Construction Main}. The twin $\ast M$ is a 4-manifold homotopy equivalent to a given 4-manifold $M$ with Kirby-Siebenmann invariant $\Ks(\ast M) \neq \Ks(M)$ \cite{[KirbySiebenmann]}. Let $M_{E_8}$ be the closed simply connected topological 4-manifold with intersection form $E_8$ \cite[Section 8.1]{[FNOP]}. If the Whitehead group of $\pi_1(M)$ and the natural homomorphism $L_4(1)\rightarrow L_4(\pi_1(M), -)$ both vanish, it is known that the twin $\ast M$ can be constructed by performing surgery to the normal map $M\cs M_{E_8}\rightarrow M$ by the exactness of the topological surgery sequence see \cite[p. 650 - 651]{[HKT]}. Another contribution of Construction \ref{Construction Main} is the construction of the homotopy equivalent twin $\ast M$ of a 4-manifold $M$ whose fundamental group need not satisfy the aforementioned conditions.\end{remark}

\section{Proof of Theorem \ref{Theorem Homeomorphism Invariants}}\label{Section Proofs}

We need to show that the statements\begin{enumerate}[(i)]
\item There is a homeomorphism $X_{2p}\rightarrow X'_{2p}$. 
\item There is a homeomorphism $X\rightarrow X'$.
\end{enumerate}are equivalent. The implication (ii)$\Rightarrow$(i) follows from the results in \cite{[BaisTorres]} as we now describe. A homeomorphism $X\rightarrow X'$ induces a homeomorphism $X\setminus \nu(\alpha_X)\rightarrow X'\setminus \nu(\alpha_{X'})$ that restricts to a homeomorphism $\partial(X\setminus \nu(\alpha_X))\rightarrow \partial(X'\setminus \nu(\alpha_{X'}))$ on the $S^2\widetilde{\times} S^1$ boundary for any odd $p > 1$. It is shown in \cite[Theorem ]{[BaisTorres]} that any such homeomorphism extends to $N_p$ used in Construction \ref{Construction Main} and obtained the following result\begin{theorem}\label{Theorem Homeomorphism}\cite[Theorem B]{[BaisTorres]}. Let $M_1$ and $M_2$ be topological 4-manifolds and let $\{\alpha_{M_i}\subset M_i: i = 1, 2\}$ be locally flat orientation-reversing simple loops. Suppose there is a homeomorphism\begin{equation*}
M_1\setminus \nu(\alpha_{M_1})\rightarrow M_2\setminus \nu(\alpha_{M_2})\end{equation*}and define\begin{equation*}M_i(p) \mathrel{\mathop:}=(M_i\setminus \nu(\alpha_{M_i}))\cup N_p.\end{equation*} for $i = 1, 2$, where $N_p$ is the compact 4-manifold in (\ref{Nonorientable Piece}). The 4-manifolds $M_1(p)$ and $M_2(p)$ are homeomorphic.
\end{theorem}Hence, the implication (ii)$\Rightarrow$(i) holds. We now show that the implication (i)$\Rightarrow$(ii) holds. If there is a homeomorphism $X_{2p}\rightarrow X'_{2p}$, then the invariants $\{\chi, w_1^4, \Ks, \Ar\}$ of $X_{2p}$ and $X_{2p}'$ coincide. Since $X$ and $X'$ are the initial data to assemble $X_{2p}$ and $X_{2p}'$, respectively, using Construction \ref{Construction Main}, Lemma \ref{Lemma Identities} and Proposition \ref{Proposition Value Pin} imply that the aforementioned set of invariants of $X$ and $X'$ coincide as well. It follows by \cite[Theorem 2]{[HKT]} that there is a homeomorphism between $X$ and $X'$. The proofs of the remaining bullet points of Theorem \ref{Theorem Homeomorphism Invariants} follow from similar arguments and the details are left to the reader.\hfill$\square$

\begin{remark}\label{Remark Stable Homeomorphism Classes}Debray \cite{[Debray]} determined the stable homeomorphism classes of 4-manifolds with finite fundamental group of order two mod 4. Explicit representatives of each of these classes are obtained through Construction \ref{Construction Main} in the cyclic fundamental group case.\end{remark}

\begin{remark}\label{Remark Cancellation}A cancellation result due to Khan \cite{[Khan]} implies that for any closed topological nonorientable 4-manifold $M_{2p}$ with $\pi_1(M_{2p})) = \Z/2p$ for $p > 1$ odd, the connected sum $M_{2p}\cs 2(S^2\times S^2)$ is homeomorphic to exactly one 4-manifold obtained from Construction \ref{Construction Main}.\end{remark}

\section{Proof of Theorem \ref{Theorem HW}:Computation of the nonorientable Hausmann-Weinberger invariant.}\label{Section Computation HW}For $p = 1$, the real projective 4-space $\RP^4$ minimizes the Euler characteristic for fundamental group of order two. Let $N$ be the quotient of $S^2\times S^2$ by the free $\Z/4$-action given by $(x, y)\mapsto (-y, x)$ for $(x, y)\in S^2\times S^2$ \cite{[HambletonHillman]}. The Euler characteristic of $N$ is $\chi(N) = 1$ and it realizes the invariant of Definition \ref{Definition HW Invariant} for $\pi_1 = \Z/4$. For $p\neq 0 \mod 2$, build the 4-manifolds closed nonorientable 4-manifolds\begin{equation}\label{Mfld1}M_{2p} = (\RP^4\setminus \nu(\alpha))\cup N_p\end{equation}and\begin{equation}\label{Mfld2}N_{4p} = (N\setminus \nu(\alpha_N))\cup N_p,\end{equation}where $N_p$ is the compact 4-manifold described in Construction \ref{Construction Main} that was built in \cite{[BaisTorres]}. A straight-forward homological computation reveals the values $\chi(M_{2p}) = 1 = \chi(N_{4p})$ and the Seifert-van Kampen theorem implies $\pi_1(M_{2p}) = \Z/2p$ and $\pi_1(N_{4p}) = \Z/4p$. These 4-manifolds realize the nonorientable Hausmann-Weinberger invariant for these groups. The remaining cases are realized by 4-manifolds constructed using (\ref{Mfld1}), (\ref{Mfld2}) and the compact 4-manifold $S_p$ of Construction \ref{Construction Main}.\hfill $\square$

\section{Proof of Theorem \ref{Theorem Candidate}.}\label{Section Smoothability}

Let $Y$ be the closed oriented topological 4-manifold constructed by Hambleton-Teichner \cite{[HambletonTeichner]} and let $\alpha_Y\subset Y$ be a locally flat simple loop that generates $\pi_1(Y) = \Z$. Put together\begin{equation}\label{Assemble Y'}Y' = (Y\setminus \nu(\alpha_Y))\cup (Mb\times S^2) = (Y\setminus \nu(\alpha_Y))\cup S_2,\end{equation}where $S_2$ is the compact 4-manifold introduced in Construction \ref{Construction Main}. The homeomorphism type of $Y'$ is unique by a result of Wang \cite{[Wang]} and we omit the choice of gluing homeomorphism in (\ref{Assemble Y'}). From the construction (\ref{Assemble Y'}), we can conclude that the identities $\chi(Y') = \chi(Y) = 4$, $\Ks(Y') = \Ks(Y)$ and $\pi_1(Y') = \pi_1(Y) = \Z$ hold. Given the decomposition $S^3\widetilde{\times} S^1 = (Mb\times S^2)\cup(S^1\times D^3)$ and the fact that a loop in a 4-manifold is determined up to isotopy by its homotopy class \cite[Example 4.1.3]{[GompfStipsicz]}, it follows that $Y'$ is homeomorphic to a connected sum $Y_0\cs(S^3\widetilde{\times} S^1)$ if and only if $Y$ is homeomorphic to $Y_0\cs(S^1\times S^3)$ for some simply connected 4-manifold $Y_0$. Since $Y$ is not homotopy equivalent to a connected sum, it follows that neither is $Y'$.

The orientation 2-cover is given by\begin{equation}\widehat{Y'} = (Y\setminus \nu(\alpha_Y))\cup (S^1\times S^2\times [-1, 1])\cup (\overline{Y\setminus \nu(\alpha_Y)}),\end{equation}which has infinite cyclic fundametal group, Euler characteristic $\chi(\widehat{Y'}) = 8$, signature $\sigma(\widehat{Y'}) = 0$, second Stiefel-Whitney class $w_2(\widehat{Y'}) \neq 0$ and Kirby-Siebenmann invariant $\Ks(\widehat{Y'}) = 0$. A result of Hambleton-Teichner \cite[Corollary 3]{[HambletonTeichner]} implies that $\widehat{Y'}$ is homeomorphic to the connected sum $4(\CP^2\cs\overline{\CP^2})\#(S^1\times S^3)$.\hfill $\square$

\begin{remark}\label{Remark Smoothable}Let $\ast \RP^4$ be the homotopy real projective 4-space with Kirby-Siebenmann invariant $\Ks(\ast \RP^4) = 1$ \cite{[Ruberman], [HKT]} and let $\ast \mathcal{E}$ be the homotopy Enrique surface with $\Ks(\ast \mathcal{E}) = 1$ \cite{[HambletonKreck]}. A classification result of Hambleton-Kreck \cite[Theorem C]{[HambletonKreck]} says that there is a homeomorphism\begin{equation*}\ast\mathcal{E}\rightarrow \overline{M}_{E_8}\cs(S^2\times S^2)\cs L_2',
\end{equation*}where $\overline{M}_{E_8}$ is the closed oriented topological simply connected 4-manifold with intersection form $-E_8$ \cite[p. 13]{[GompfStipsicz]} and $L_2'$ is a $\Q$-homology 4-sphere with fundamental group of order two, second Stiefel-Whitney class $w_2(L_2')\neq 0$ and universal cover $S^2\times S^2$ \cite[Figure 5.46]{[GompfStipsicz]}. It is stated in \cite[Problem 4.20, Remark (4)]{[KList]} that it would be interesting to know if the connected sum $\ast \RP^4\cs \ast \mathcal{E}$ admits a smooth structure. To see that it does, apply Hambleton-Kreck-Teichner's result \cite[Theorem 2]{[HKT]} to obtain homeomorphisms \begin{equation}\label{Homeomorphisms}\ast \RP^4\cs\ast \mathcal{E}\rightarrow \ast \RP^4\cs\overline{M_{E_8}}\cs(S^2\times S^2)\cs L_2'\rightarrow \RP^4\cs 5(S^2\times S^2)\cs L_2'
\end{equation}and observe that the 4-manifold on the right hand side of (\ref{Homeomorphisms}) is smooth.\end{remark}

\end{document}